\documentclass{article}
\usepackage[T1]{fontenc}
\usepackage[latin1]{inputenc}

\usepackage{amsmath,amsthm,amssymb,mathtools,array,centernot,paralist}

\usepackage{tikz}

\newtheorem{theorem}{Theorem}[section]
\newtheorem{proposition}{Proposition}[section]

\newtheorem{lemma}[theorem]{Lemma}
\newtheorem{corollary}[theorem]{Corollary}
\theoremstyle{definition}

\newtheorem{definition}{Definition}[section]
\newtheorem*{remark}{Remark}

\newcommand{\Z}{\ensuremath{\mathbb{Z}}}
\renewcommand{\P}{\mathbf{P}}

\DeclareMathAlphabet{\matris}{T1}{cmss}{m}{sl}

\newcommand{\E}{\mathcal{E}}
\newcommand{\tildeE}{\widetilde{\mathcal{E}}}
\newcommand{\negspace}{\mkern-18mu}

\title{Counting linear congruence systems with a fixed number of solutions}
\author{%
  Marcus Nilsson\\
  \texttt{marcus.nilsson@lnu.se}\\
  Linnaeus University\\
  Sweden
}

\begin{document}

\maketitle

\begin{abstract}
	For a prime $p$ and a positive integer $s$ consider a homogeneous linear system over the ring $\Z_{p^s}$ (the ring of integers modulo $p^s$) described by an $n\times m$-matrix. The possible number of solutions to such a system is $p^j$, 
	where $j=0,1,\ldots, sm$. We study the problem
	of how many $n\times m$-matrices over $\Z_{p^s}$ there are given that we have exactly $p^j$ homogeneous solutions. 
In a more algebraic language we count the number of homomorphisms $A:\Z_{p^s}^m\to \Z_{p^{s}}^n$ such that $|\ker(A)|=p^j$.
	
For the case $s=1$ (when $\Z_{p^s}$ is a field) George von Landsberg proved a general formula in 1893,\cite{Landsberg1893}. However, there seems to be few published general results for the case $s>1$ except when we have a unique solution ($j=0$). In this article we present recursive methods for counting such matrices and present explicit formulas
for the case when $j\leq s$ and $n\geq m$. We will use a generalization of Euler's $\varphi$-function and Gaussian binomial coefficients to express our formulas. 	
	
As an application we compute the probability that
$\gcd(\det(A),p^s)$ gives the number of solutions to the quadratic system $Ax=0$ in $\Z_{p^s}^n$.
\end{abstract}

\section{Introduction}
For a prime $p$ and a positive integer $s$ consider a homogeneous linear system over the ring $\Z_{p^s}$ (the ring of integers modulo $p^s$) $Ax=0$, where $A$ is an $n\times m$-matrix. The possible number of solutions to such a system is $p^j$, where $j=0,1,\ldots, sm$. To solve such a system there are several alternatives. On can for example use Smiths normal form, first published in \cite{Smith1861}, and rewrite the system to a diagonal system. One can also use ordinary Gaussian elimination with possible use of reductions to rings $\Z_{p^t}$ for $t<s$ to find invertible elements, see for example \cite{NilssonNyqvist2009}.  If we know the solutions modulo $p^s$ we can by using the Chinese remainder theorem easily solve the system modulo any integer $N\geq 2$ if we know its prime factorization. 

In this article we study the following inverse problem:\\ 
How many $n\times m$-matrices over $\Z_{p^s}$ are there given that we have exactly $p^j$ homogeneous solutions, for $j\in \{0,1,\ldots, s m\}$? In a more algebraic language we want to count the number of module homomorphisms $A:\Z_{p^s}^m\to \Z_{p^{s}}^n$ such that $|\ker(A)|=p^j$. We let $\E(n\times m,p^s,p^j)$ denote this number.	

For the case $s=1$ (when $\Z_{p^s}$ is a field) George von Landsberg proved a general formula in 1893,\cite{Landsberg1893}.
He proved that 
\begin{equation}\label{eq:Landsberg}
		\E(n\times m, p,p^j)=\binom{m}{j}_p \prod_{i=0}^{m-j-1}(p^n-p^i),
\end{equation}
where $\binom{m}{j}_p$ is a so called Gaussian binomial coefficient. 
However, there seems to be few published general results for the case $s>1$ except when we have a unique solution ($j=0$). 

According to \cite{Overbey2005}, when discussing possible number of Hill ciphers, a formula for the number of invertible square $n\times n$ matrices is 
\begin{equation}\label{eq:Overbey} 
	\E(n\times m, p^s,1)=p^{(s-1)n^2}\prod_{u=0}^{n-1}(p^n-p^u).
\end{equation}
The same formula can also be found in \cite{Han2006}.
Results on counting matrices over finite fields can also be found in for example \cite{Laksov1994}.

In this article we present recursive formulas for counting matrices over $\Z_{p^s}$ using generalized Euler $\varphi$-functions. We will also present explicit formulas for the case when $j\leq s$ and $n\geq m$. 

In Section \ref{sec2} we recall the definition and formulas for Gaussian binomial coefficients, defined by Gauss in \cite{Gauss1808}. We also define a generalization of Euler's $\varphi$-function. 
In Section \ref{sec3} we introduce further notations and derive recursive relations.  	
In Section \ref{sec4} we present the explicit formula for $\E(n\times m,p^s,p^j)$ for $j<s$ and in Section \ref{sec5} we handle the case when $j=s$.

Section \ref{sec6} contains an investigation of the sum $\sum_{j=0}^s \E(n\times n,p^s,p^j)$ in the quadratic case. We apply the result to estimate the probability that the formula $\gcd(\det(A),p^s)$ gives the correct number of solutions to the system.
\section{Preliminaries}\label{sec2}

We recall the definition of the Gaussian binomial coefficients (or $q$-binomial coefficients). For details we refer to a books in combinatorics, for example \cite{Stanley1997}. 
\begin{definition}
	Let $n$ and $k$ be non-negative integers. If $k\leq n$ we define
	\begin{equation}\label{eq:qbinom}
	\binom{n}{k}_q=\frac{(1-q^n)(1-q^{n-1})\cdots (1-q^{n-k+1})}{(1-q)(1-q^2)\cdots (1-q^{k})}.
	\end{equation}
	If $k=0$ we let $\binom{n}{k}_q=1$. If $k>n$ we set $\binom{n}{k}_q=0$.
\end{definition}
We also have the following analog of Pascal's identity
\begin{equation}\label{pascal}
\binom{n}{k}_q=\binom{n-1}{k}_q+q^{n-k}\binom{n-1}{k-1}_q.
\end{equation}
It can also be written as
\begin{equation}\label{pascal2}
\binom{n}{k}_q=\binom{n-1}{k-1}_q+q^{k}\binom{n-1}{k}_q.
\end{equation}
By induction we can easily prove that $\binom{n}{k}_q$ is a polynomial in $q$ by repeated use of \eqref{pascal}.
We can also easily prove that
\begin{equation}\label{q-degree}
	\deg(\binom{n}{k}_q)=k(n-k).
\end{equation}
We end the revision on Gaussian binomial coefficients with two formulas
\begin{equation}\label{eq:q-bin_sum}
	\sum_{i=0}^{j} p^i\binom{n+i}{i}_q=\binom{n+j+1}{j}_q,
\end{equation} 
and
\begin{equation}\label{eq:q_bin_prod}
	\sum_{k=0}^h\binom{n}{k}_q\binom{m}{h-k}_q q^{(n-k)(h-k)}=\binom{m+n}{h}_q
\end{equation}
The first is proved by induction and the latter is the so called $q$-analogue of Vandermonde's identity.
For the proof and further details we again refer to \cite{Stanley1997}.

In this article we will use a generalization of Euler's $\varphi$-function.
\begin{definition}
	For a non-negative integer $c$ let $\varphi_n(c)$ denote the number of $n$-tuples $(a_1,a_2,\ldots, a_n)$ where $1\leq a_i\leq c$ and at least one $a_i$ is relative prime to $c$.
\end{definition}
Let $p$ be a prime number. For $c=p^s$ we have
\[
	\varphi_n(p^s)=(p^{s})^n-(p^{s-1})^{n}=(p^{n(s-1)})(p^n-1),
\]
for $s\geq 1$. For $s=0$, $\varphi_n(p^0)=\varphi_n(1)=1$.

Note that for $n=1$ we have the usual Euler's $\varphi$-function.
But, note also that $\varphi_n$ does not inherit all the properties of the $\varphi$-function. For example $\varphi_n$ is not multiplicative for $n > 1$.

\section{Recursive formulas}\label{sec3}
In this section we will derive recursive formulas for $\E(n\times m,p^s,p^j)$. To do this we will need to look at some specific subsets of the set of $n\times m$-matrices having $p^j$ homogeneous solutions.

\begin{definition}Let $\tildeE(n\times m, p^s,p^j)$ denote the number of matrices $A$ having at least one invertible element modulo $p^s$ and giving $p^j$ homogeneous solutions to the system $Ax\equiv 0\pmod{p^s}$. We will call such matrices \emph{relatively prime}.
\end{definition}
\begin{definition}
Let $\tildeE_i(n\times m, p^s,p^j)$ denote the number of relatively prime matrices (giving $p^j$ solutions) where the first invertible element occurs in the $i+1$ column. The first $i$ columns will all be divisible by $p$.
\end{definition}
We can now count the relatively prime matrices by using $\tildeE_i$.Note that the first $i$ columns will all be divisible by $p$.

\begin{proposition}\label{prop:tildeEsum}
		Let $A$ be a matrix where the first $i$ columns are divisible by $p$, then the system $Ax\equiv 0\pmod{p^s}$ has at least $p^i$ solutions. Moreover
		\[
			\tildeE(n\times m, p^s,p^j)=\sum_{i=0}^{\min(j,m-1)}\tildeE_i(n\times m, p^s,p^j).
		\]
\end{proposition}
\begin{proof}
	It is clear that
	\[
		\tildeE(n\times m, p^s,p^j)=\sum_{i=0}^{m-1}\tildeE_i(n\times m, p^s,p^j).
	\]
	We will now show that $\tildeE_i(n\times m, p^s,p^j)=0$ for $i>j$. 
	We can write the system as
	\[
		A_1 x_1+ A_2 x_2+\ldots + A_i x_i+\ldots + A_m x_m\equiv 0\pmod{p^s},
	\]
	where $A_l$ denotes a column of $A$. 
	Let $1\leq l\leq i$, then $A_l=p B_l$ for some $B_l\in \Z_{p^{s-1}}^n$. For a given solution with coordinates $x_l$ also $x_l+c p^{s-1}$, where $c=0,1,\ldots p-1$, is also a solution since $A_l (x_l+c p^{s-1})\equiv A_l x_1+B_l c p^s\equiv A_l x_l\pmod{p^s}$. Hence there are at least p choices for $x_l$ for $1\leq l\leq i$, so the system has at least $p^i$ solutions.
	Hence, for $i>j$ we must have $\tildeE_i(n\times m,p^s,p^j)=0$. 
\end{proof}

There is also a simple formula connecting $\E(n\times m,p^s,p^j)$ and $\tildeE(n\times m,p^s,p^j)$.
\begin{proposition}\label{prop:p_reduce}
	We have
	\[
	\E(n\times m,p^s,p^j)=\E(n\times m, p^{s-1},p^{j-m})+\tildeE (n\times m,p^s,p^j).
	\]
\end{proposition}
\begin{proof}
	Assume that $A$ is an $n\times m$-matrix such that $Ax\equiv 0\pmod{p^s}$ has $p^j$ solutions.
	Two cases can occur. Either $A$ is a relatively prime matrix that is $\gcd(A,p^s)=\gcd(A,p)=1$ or $\gcd(A,p^s)\geq p$. If $\gcd(A,p)=1$ we count the matrix in $\tildeE (n\times m,p^s,p^j)$. If $\gcd(A,p^s)\geq p$ we let $A=p B$ for some matrix $B$ with entries in the set $\Z_{p^{s-1}}$. By the same arguments as in the proof of Proposition \ref{prop:tildeEsum} a solution to $Bx\equiv 0\pmod{p^{s-1}}$ can be lifted to $p^m$ solutions in $Ax\equiv 0\pmod{p^s}$. So, the matrix will be counted in 
	$\E(n\times m, p^{s-1},p^{j-m})$.
\end{proof}

\begin{remark}
In the case where $j<m$ (the dimension of the solution is less that the number of columns) we have
$\E(n\times m, p^s,p^j)=\tildeE(n\times m, p^s,p^j)$.
\end{remark}

By repeated use of Proposition \ref{prop:p_reduce} we get the following result.
\begin{proposition}\label{prop:gcd_reduce}
	With the above notations
	\[
		\E(n\times m,p^s,p^j)=\negspace\sum_{k=\max(0,j-s(m-1))}^{\lfloor \frac{j}{m}\rfloor}\negspace\widetilde{\E}(n\times m,p^{s-k},p^{j-mk}).
	\]
\end{proposition}
\begin{proof}
	We can also prove this directly. 
	Let $d=p^k=\gcd(A,p^s)$. Then the system $Ax\equiv 0\pmod{p^s}$ have at least $p^{mk}$ solutions. This corresponds to a unique solution modulo $p^{s-k}$ that lifts to $p^{km}$ solutions modulo $p^s$. (In a similar way as in the proof of Proposition \ref{prop:tildeEsum}).
	
	At most we can have $p^{s(m-1)+k}$ solutions. This corresponds to a matrix where all but one columns are identically 0 and one column is exactly divisible by $p^k$.  Hence, for given $k$ we have the following limits of the number of solutions $p^j$
	\[
		m k\leq j\leq s(m-1)+k.
	\]
	For a given $j$ we have $k\leq \frac{j}{m}$ and therefore $k\leq \lfloor\frac{j}{m}\rfloor$.
	We also have $j-s(m-1)\leq k$ so $k\geq \max(0, j-s(m-1))$. To sum up we have
	\[
		\max(0,j-s(m-1))\leq k\leq \left\lfloor\frac{j}{m}\right\rfloor.\qedhere
	\]
\end{proof}

We now want to find a formula for $\tildeE(n\times m,p^s,p^j)$ in terms of the number of matrices of type $(n-1)\times (m-1)$. 
Let $A=(a_{kl})$ be an $n\times m$-matrix such that $\gcd(A,p^s)=1$ and where the first $i$ columns are divisible by $p$ and an invertible element exists in column $i+1$. We can then without restriction assume that $a_{1(i+1)}$ is invertible. We also have
$a_{kl}=p b_{kl}$ for some $b_{kl}\in \Z_{p^{s-1}}$ for $1\leq k \leq n$ and $1\leq l\leq i$. By using Gaussian elimination on $Ax\equiv 0\pmod{p^s}$ with $a_{1(i+1)}$ as pivot we get the equivalent system $\widehat{A}x\equiv 0\pmod{p^s}$ where $\widehat{A}$ is given by

\begin{equation}\label{eq:red_matrix}
\begin{pmatrix}
	p b_{11} & \cdots & p b_{1i} & a_{1(i+1)} & a_{1(i+2)} & \cdots & a_{1m}\\
	c_{21} & \cdots & c_{2i} & 0 & c_{2(i+2)} & \cdots &c_{2m}\\
	\vdots & &\vdots & \vdots & \vdots & & \vdots\\
	c_{n1} & \cdots & c_{ni} & 0 & c_{n(i+2)} & \cdots & c_{nm}
\end{pmatrix},
\end{equation}
where
\[
	c_{jk}=
	\begin{cases}
		-(a_{1(i+1)})^{-1} a_{j(i+1)}p b_{1k}+p b_{jk}\mod p^s,\ k\leq i,\\
		-(a_{1(i+1)})^{-1} a_{j(i+1)}a_{1k}+ a_{jk}\mod p^s,\ k\geq i+2.
	\end{cases}
\]
Note that $p\mid c_{jk}$ for all $j$ and $k\leq i$.
Hence, $C=(c_{jk})$ is an $(n-1)\times (m-1)$-matrix with at least the first $i$ columns not being invertible.

We denote by $\E_i(n\times m,p^s,p^j)$ the number of matrices giving $p^j$ solutions, such that (at least) the first $i$ columns of the matrix are divisible by $p$. The following lemma relates $\E_i$ and $\tildeE_k$ for $k\geq i$.
\begin{lemma}\label{lma:reducing1}
	Let $0 \leq i\leq m-1$. We have
	\[
		\E_i(n\times m, p^s,p^j)=\E(n\times m,p^{s-1}, p^{j-m})+\sum_{k=i}^{\min(j,m-1)}\tildeE_k(n\times m,p^s,p^j).
	\]
	We also have $\E_i(n\times m, p^s,p^j)=0$ for $j< i$.
\end{lemma}
\begin{proof}
	From start we have matrices where the first $i$ columns are divisible by $p$, and they give $p^j$ solutions.
	By the Propositions \ref{prop:p_reduce} and \ref{prop:tildeEsum} we get
	\[
		\E_i(n\times m, p^s,p^j)=\E(n\times m,p^{s-1}, p^{j-m})+\sum_{k=0}^{\min(j,m-1)}\tildeE_k(n\times m,p^s,p^j).
	\]
	If there is an invertible element, the leftmost one must occur in one of the columns $i+1$ to $m$ and therefore $\tildeE_k(n\times m,p^s,p^j)=0$ for $k<i$.
	
	The second statement of the lemma follows by the same argument  as in the proof of Proposition \ref{prop:tildeEsum} that $\tildeE_i(n\times m,p^s,p^j)=0$ for $i>j$.
\end{proof}

We now try to find an expression for $\tildeE_i$ of $n\times m$-matrices in terms of $\E$ and $\E_i$ of matrices of type $(n-1)\times (m-1)$.  
\begin{lemma}\label{lma:reducing2}
Let $i\leq j$. For $i<m-1$ we have 
\[
	\tildeE_i(n\times m,p^s,p^j)=\varphi_n(p^s)(p^{s-1})^i (p^s)^{m-(i+1)}\E_i((n{-}1){\times}(m{-}1),p^s,p^j).
\]
For $i=m-1$ we have
\[
\tildeE_{m-1}(n\times m,p^s,p^j)=\varphi_n(p^s)(p^{s-1})^{m-1}\E((n{-}1){\times}(m{-}1),p^{s-1},p^{j-(m-1)}).
\]
\end{lemma}
\begin{proof}
	We consider systems $Ax\equiv 0\pmod{p^s}$ where $A$ is an $n\times m$-matrix, where the first $i$ columns are non-invertible and there exists an invertible element in column $(i+1)$. We can perform Gaussian elimination like in \eqref{eq:red_matrix}. We can also assume that $a_{1(i+1)}$ is invertible. 
	
	First assume that $i<m-1$. Let $C$ be the matrix from \eqref{eq:red_matrix}. Then the system has $p^j$ solutions if and only if the system $Cx\equiv 0\pmod{p^s}$ has $p^j$ solutions. In this case we know the first $i$ columns of $A$ are divisible by $p$ (but maybe not the $m-(i+1)$ last columns). Hence, in this case we can choose matrix A in 
	\[
		\varphi_n(p^s)(p^{s-1})^i (p^s)^{m-(i+1)}\E_i((n-1)\times (m-1),p^s,p^j)
	\]
	ways. In fact we choose the $(i+1)$-column in $\varphi_n(p^s)$ ways, each of the $i$ elements $b_{11}$ to $b_{1i}$ in $p^{s-1}$ ways and each of the $m-(i+1)$ elements $a_{1(i+2)}$ to $a_{1m}$ in $p^s$ ways.
	
	If $i=m-1$ then all the columns of the matrix $C$ are divisible by $p$. The system  $Ax\equiv 0\pmod{p^s}$ has $p^j$ solutions if the system $Cx\equiv 0\pmod{p^{s}}$ has $p^{j-(m-1)}$ solutions modulo $p^{s-1}$. (Each solution modulo $p^{s-1}$ can be lifted to $p^{m-1}$ solutions modulo $p^s$.) Hence, in the $i=m-1$ case we can choose the matrix in
	\[
		\varphi_n(p^s) (p^{s-1})^{m-1}\E((n-1)\times (m-1),p^{s-1},p^{j-(m-1)})
	\]
	ways.
\end{proof}

We end this section by combining Lemma \ref{lma:reducing1} and Lemma \ref{lma:reducing2} and directly get a recursive relations for $\tildeE_i$ that will be used in the derivation of the explicit formula in next section.
\begin{proposition}\label{prop:recursion_i}
	For $i\leq m-1$, we have 
	\begin{multline*}
		\tildeE_{i}(n\times m,p^s,p^{j})=\varphi_n(p^s)(p^{s-1})^i (p^s)^{m-(i+1)}\\
		\times \left(
	\E((n{-}1){\times}(m{-}1),p^{s-1},p^{j-(m-1)})+\sum_{k=i}^{\min(j,m-2)} \tildeE_k((n{-}1){\times}(m{-}1),p^s,p^{j})\right).
	\end{multline*}
	For $j<m-1$,  we set $\E((n{-}1){\times}(m{-}1),p^{s-1},p^{j-(m-1)})=0$ and for $i=m-1$ the sum is empty.
	As always, $\tildeE_i(n\times m,p^s,p^j)=0$ if $j<i$.
\end{proposition}

\section{Explicit formulas for $j<s$}\label{sec4}
In this section we will find for $n\geq m$ an explicit formula for the number of $n\times m$-matrices $A$ over $\Z_{p^s}$ such that
$Ax\equiv 0\pmod{p^s}$ has exactly $p^j$ solutions for $j<s$. The case $j=s$ will be considered in the next section. 
In the proof we will reduce to smaller and smaller systems (by using Proposition \ref{prop:recursion_i}) finally ending in a system with an $n\times 1$-matrix. 

\begin{proposition}\label{prop:nx1}
The number of matrices of type $n\times 1$ that gives $p^j$ homogeneous solutions is
\[
	\E(n\times 1, p^s,p^j)=\varphi_n(p^{s-j}).
\]
If $j=s$ we have $\E(n\times 1, p^s,p^s)=\varphi_n(1)=1$ and 
if $j>s$ we have $\E(n\times 1, p^s,p^j)=0$.
\end{proposition}
\begin{proof}
For the case $j=s$ there is only one possible matrix, the zero matrix, so $\E(n\times 1, p^s,p^s)=1$.
For the case $j>s$ we have $\E(n\times 1, p^s,p^j)=0$ since there are at most $p^s$ solutions to the the system.
Let $A$ be an $n\times 1$-matrix
and consider the system
\begin{equation}\label{eq:n1}
	Ax\equiv 0\pmod{p^s}.
\end{equation}
Let $d=\gcd(A,p^s)$. We have $d=p^k$ for some $k\in\{0,1,\ldots, s-1\}$.
If we divide all the equations in \eqref{eq:n1} by $p^k$ we get
\begin{equation*}
\left\{
\begin{array}{r@{\;}c@{\;}l}
	b_1 x &\equiv& 0\pmod{p^{s-k}}\\
	b_2 x &\equiv& 0\pmod{p^{s-k}}\\
	&\vdots&\\
	b_n x &\equiv& 0\pmod{p^{s-k}},
	\end{array}
	\right.
\end{equation*}
where $a_i=b_i p^k$ and $0\leq b_i\leq p^{s-k}-1$. At least one of the $b_i$ must be relatively prime to $p$.
The corresponding equation has a unique solution ($x=0$) modulo $p^{s-k}$ and $p^k$ solutions modulo $p^s$. Note that all these solutions also solve all the other equations in the system \eqref{eq:n1}. So the system have $p^j$ solutions.

Hence, we can choose $A$ in $\varphi_n(p^{s-j})$ ways to get $p^j$ solutions.
This is the number of ways can we choose the $b_i$:s such that at least one of them is relatively prime to $p$.
\end{proof}

As an example how we can use the recursive results from last section we prove a formula for $\tildeE(n\times 2, p^s,p^j)$.
\begin{proposition}\label{prop:nx2}
	We have
	\[
		\tildeE(n\times 2,p^s,p^j)=
		\begin{cases}
		\varphi_n(p^s)\varphi_{n-1}(p^s)p^s,\ j=0\\
			\varphi_n(p^s)\varphi_{n-1}(p^{s-j})(p^s+p^{s-1}),\ 1\leq j\leq s\\
			 0,\ j>s.
		\end{cases}
	\]
\end{proposition}
\begin{proof}
	First consider the case where $j=0$ (unique solution). 
	From Proposition \ref{prop:tildeEsum} we have
	\[
		\tildeE(n\times 2, p^s,1)=\sum_{i=0}^{\max(j,m-1)}{\mkern-18mu}\tildeE_i(n\times 2, p^s,1)=\tildeE_0(n\times 2, p^s,1).
	\]
	From Proposition \ref{prop:recursion_i}  and Proposition \ref{prop:nx1} we have
	\[
		\tildeE_0(n\times 2, p^s,1)=\varphi_n(p^s)p^s\tildeE((n-1)\times 1, p^s,1)=\varphi_n(p^s)\varphi_{n-1}(p^s) p^s.
	\]
	
	Let us now assume that $1\leq j\leq s$. From Proposition \ref{prop:tildeEsum} we have
	\[
		\tildeE(n\times 2, p^s,p^j)=\tildeE_0(n\times 2, p^s,p^j)+\tildeE_1(n\times 2,p^s,p^j).
	\]
	Again from Proposition \ref{prop:recursion_i} we have
	\[
		\tildeE_0(n\times 2,p^s,p^j)=\varphi_n(p^s)p^{s}
		\left(\E((n-1)\times 1,p^{s-1},p^{j-1})+\tildeE_0((n-1)\times 1,p^s,p^j)\right).
	\]
	From Proposition \ref{prop:nx1} we get
	\[
	\E((n-1)\times 1,p^{s-1},p^{j-1})=\varphi_{n-1}(p^{s-1-(j-1)})=\varphi_{n-1}(p^{s-j}).
	\]
	We also have $\tildeE_0((n-1)\times 1,p^s,p^j)=0$ since an $(n-1)\times 1$ matrix with at least one element relatively prime to $p$ only has the trivial solution. Hence, we can conclude that
	\[
		\tildeE_0(n\times 2,p^s,p^j)=\varphi_n(p^s)\varphi_{n-1}(p^{s-j})p^{s}.
	\]
	One last application of Proposition \ref{prop:recursion_i} and Proposition \ref{prop:nx1} yields
	\begin{multline*}
		\tildeE_1(n\times 2,p^s,p^j)=\varphi_n(p^s)p^{s-1}\E((n-1)\times 1, p^{s-1},p^{j-1})\\
		=\varphi_n(p^s)\varphi_{n-1}(p^{s-j})p^{s-1}.
	\end{multline*}
	Hence, for $1\leq j\leq s$,
	\[
		\tildeE(n\times 2,p^s,p^j)=\varphi_n(p^s)\varphi_{n-1}(p^{s-j})(p^{s-1}+p^s).
	\]
	If $j>s$ then $\tildeE(n\times 2,p^s,p^j)=0$ since the number of solutions to such a system can be at most $p^s$.
\end{proof}

To find a formula for $\E(n\times 2, p^s,p^j)$ one possibility is to use Proposition \ref{prop:p_reduce} and induction.
We will instead use Proposition \ref{prop:gcd_reduce}.
\begin{proposition}\label{prop:nx2:gen}
	We have for $j<s$
	\[
		\E(n\times 2,p^s,p^j)=\varphi_n(p^s)\varphi_{n-1}(p^{s-j})p^{s-j}\binom{j+1}{1}_p.
	\]
\end{proposition}
\begin{proof}
		From Proposition \ref{prop:gcd_reduce} we have
	\[
		\E(n\times 2,p^s,p^j)=\sum_{k=\max(0,j-s)}^{\lfloor \frac{j}{2}\rfloor}\mkern-18mu\tildeE(n\times 2,p^{s-k},p^{j-2k}).
	\]
	If $j$ is even ($j=2a$) then
from Proposition \ref{prop:nx2} 
	\begin{multline*}
		\E(n\times 2,p^s,p^j)=\sum_{k=0}^{a}\tildeE(n\times 2,p^{s-k},p^{j-2k})\\
		=\sum_{k=0}^{a-1}\varphi_n(p^{s-k})\varphi_{n-1}(p^{s-j+k})(p^{s-k}+p^{s-k-1})
		+\varphi_n(p^{s-a})\varphi_{n-1}(p^{s-a})p^{s-a}\\
		=\sum_{k=0}^{a-1}\varphi_n(p^{s})\varphi_{n-1}(p^{s-j})(p^{s-2k}+p^{s-2k-1})
		+\varphi_n(p^{s})\varphi_{n-1}(p^{s-j})p^{s-2a}\\
		=\varphi_n(p^s)\varphi_{n-1}(p^{s-j})\left(p^s+p^{s-1}+\ldots+ p^{s-j}\right)\\
		=\varphi_n(p^s)\varphi_{n-1}(p^{s-j})p^{s-j}\binom{j+1}{1}_p.
	\end{multline*}
	Here we used the gaussian binomial coefficient defined in \eqref{eq:qbinom}.
	If $j$ is odd ($j=2a+1$) a similar computation gives the same result.
\end{proof}

For the coming proof by induction of the general theorem we will also need the formula for $j=0$. This result can also be found in for example \cite{Overbey2005} and \cite{Han2006}, see the remark after the theorem. 
\begin{theorem}\label{thm:unique}
	Let $p$ be a prime and let $s\in\mathbb{N}$. 
	The number of $n\times m$ matrices $A$ such that $Ax\equiv 0\pmod{p^s}$ has a unique solution is
	\[
		\E(n\times m,p^s,1)=(p^{s})^{\frac{m(m-1)}{2}}\mkern-18mu\prod_{u=n-m+1}^{n}\mkern-18mu\varphi_u(p^s).
	\]
\end{theorem}
\begin{proof}
	We will prove this by induction on $m$.
	Since we are only looking for matrices that give us unique solution we note that
	\[
		\E(n\times m,p^s,1)=\tildeE(n\times m,p^s,1)=\tildeE_0(n\times m,p^s,1).
	\]

	From Proposition \ref{lma:reducing2} we get
	\begin{equation}\label{eq:j1:reduce}
		\tildeE_0(n\times m,p^s,1)=\varphi_n(p^s)(p^s)^{m-1}\tildeE_{0}((n-1)\times (m-1),p^s,1).
	\end{equation}
	From Proposition \ref{prop:nx1} we have
	\[
		\E(u\times 1,p^s,1)=\tildeE_0(u\times 1,p^s,1)=\varphi_u(p^{s-0})=\varphi_u(p^s)
	\]
	for all positive integers $u$. This also proves the theorem for $m=1$.
	
	We now assume that
	\[
		\tildeE_{0}((n-1)\times (m-1),p^s,1)=p^{s\frac{(m-1)(m-2)}{2}}\mkern-18mu\prod_{u=n-m+1}^{n-1}\mkern-18mu\varphi_u(p^s),
	\]
	where $2\leq m\leq n$. Then from \eqref{eq:j1:reduce},
	\begin{multline*}
		\tildeE_{0}(n\times m,p^s,1)=\varphi_n(p^s)(p^s)^{m-1}p^{s\frac{(m-1)(m-2)}{2}}\mkern-18mu\prod_{u=n-m+1}^{n-1}\mkern-18mu\varphi_u(p^s)\\
		=p^{s\frac{m(m-1)}{2}}\mkern-18mu\prod_{u=n-m+1}^{n}\mkern-18mu\varphi_u(p^s).
	\end{multline*}
\end{proof}
 \begin{remark}\label{rmk:pform}
	We can easily rewrite the formula in last theorem in the following way
	\begin{equation*}
		\E(n\times m,p^s,1)=p^{(s-1)\frac{n(n+1)}{2}+s\frac{m(m-1)}{2}}\mkern-18mu\prod_{u=n-m+1}^n \mkern-18mu (p^u-1).
	\end{equation*}
	For the case $n=m$ we have
	\[
		\E(n\times n,p^s,1)=p^{s(n^2-1)}\prod_{u=1}^n(p^u-1)
	\]
	This is the how the formula is written in \cite{Overbey2005}.
\end{remark}
We are now ready to prove the main theorem of this section.

\begin{theorem}\label{thm:main_sgj}
Assume that $n\geq m$ and that $s>j$ then
\begin{multline*}
	\E(n\times m,p^s,p^j)\\
	=\varphi_n(p^s)\cdots \varphi_{n-m+2}(p^s)\varphi_{n-m+1}(p^{s-j})p^{s\frac{m(m-1)}{2}-j(m-1)}\binom{m+j-1}{j}_p.
\end{multline*}
\end{theorem}

\begin{proof}
	For $m=1$ it follows from Proposition \ref{prop:nx1} that the formula is true. Is also follows from Proposition \ref{prop:nx2:gen} that the formula is true for $m=2$. Theorem \ref{thm:unique} is the case $j=0$. 
	For $j<m$ we have
	\[
		\E(n\times m,p^s,p^j)=\tildeE(n\times m,p^s,p^j).
	\]
	From Proposition \ref{prop:p_reduce} we have the recursive relation
	\[
		\E(n\times m, p^s,p^j)=\E(n\times m, p^{s-1},p^{j-m})+\tildeE(n\times m, p^s,p^j).
	\]

	 We will prove the formula for $\E$ by induction on $m$. In parallel we will also for $0\leq i\leq m-1$ show
\begin{multline*}
	\tildeE_i(n\times m,p^s,p^j)=\varphi_n(p^s)\cdots \varphi_{n-m+2}(p^s)\varphi_{n-m+1}(p^{s-j})\\
	\times p^{s\frac{m(m-1)}{2}-j(m-1)}p^{j-i}\binom{m-1+(j-i)-1}{j-i}_p
\end{multline*}

	As already said the case $m=1$ is proces in Proposition \ref{prop:nx1}. For $m=2$ we have from the proof of Proposition \ref{prop:nx2} that
	\begin{align*}
		\tildeE_0(n\times 2,p^s,p^j)&=\varphi_n(p^s)\varphi_{n-1}(p^{s-j})p^{s},\\
		\tildeE_1(n\times 2,p^s,p^j)&=\varphi_n(p^s)\varphi_{n-1}(p^{s-j})p^{s-1},
	\end{align*}
	and from Proposition \ref{prop:nx2:gen}
	\[
		\E(n\times 2,p^s,p^j)=\varphi_n(p^s)\varphi_{n-1}(p^{s-j})p^{s-j}\binom{j+1}{1}_p.
	\]
	We now assume for $0\leq i\leq m-2$,  
	\begin{multline*}
		\tildeE_i((n{-}1)\times (m{-}1),p^s,p^j)
		=\varphi_{n-1}(p^s)\cdots \varphi_{n-m+2}(p^s)\varphi_{n-m+1}(p^{s-j})\\
		\times p^{s\frac{(m-1)(m-2)}{2}-j(m-2)}p^{j-i}\binom{m-1+(j-i)-2}{j-i}_p.
	\end{multline*}
	Note that for $i=j$ we have $\binom{m+(j-i)-2}{j-i}_p=1$ and for $i>j$ we set $\binom{m+(j-i)-2}{j-i}_p=0$. We also assume
	\begin{multline*}
		\E(n\times k,p^s,p^j)
	=\varphi_n(p^s)\cdots \varphi_{n-k+2}(p^s)\varphi_{n-k+1}(p^{s-j})\\
	\times p^{s\frac{k(k-1)}{2}-j(k-1)}\binom{k+j-1}{j}_p,
	\end{multline*}
	for $k=m-1$ and any values of $n$, $s$ and $j$ satisfying the conditions $n\geq m$ and $s>j$.
	We get from Proposition \ref{prop:recursion_i}
	\begin{multline*}
		\tildeE_{m-1}(n\times m,p^s,p^j)=\varphi_n(p^s)(p^{s-1})^{m-1}\E((n{-}1){\times}(m{-}1),p^{s-1},p^{j-(m-1)})\\
		=\varphi_n(p^s)p^{s(m-1)-(m-1)}\varphi_{n-1}(p^s)\cdots\varphi_{n-m+2}(p^{s-j})p^{s\frac{(m-1)(m-2)}{2}-j(m-2)}
	\binom{j-1}{m}_p\\
	=\varphi_{n}(p^s)\cdots\varphi_{n-m+1}(p^{s-j})p^{s\frac{m(m-2)}{2}-j(m-2)-(m-1)}\binom{j-1}{m}_p.
	\end{multline*}
	Here we have used 
	\begin{multline}\label{eq:oldlemma}
		\E((n{-}1)\times (m{-}1),p^{s-1},p^{j-(m-1)})=\\
		\varphi_{n-1}(p^s)\cdots\varphi_{n-m+2}(p^{s-j})p^{s\frac{(m-1)(m-2)}{2}-j(m-2)}\binom{j-1}{m-2}_p
	\end{multline}
	that follows from the properties of the $\varphi$-functions.
	
	For $j<m-1$ we get
	\begin{multline*}
		\tildeE_i(n\times m,p^s,p^j)=\varphi_n(p^s)(p^{s-1})^i )(p^{s})^{m-1-i}
		\sum_{k=i}^j\tildeE_k((n{-}1)\times(m{-}1),p^s,p^j)\\
		=\varphi_{n}(p^s)\cdots \varphi_{n-m+2}(p^s)\varphi_{n-m+1}(p^{s-j})p^{s(m-1)-i}p^{s\frac{(m-1)(m-2)}{2}-j(m-2)}\\
		\times \left(1+p\binom{m-2}{1}_p+\ldots+p^{j-i}\binom{m-1+(j-i)-2}{j-i}_p\right)\\
		=\varphi_{n}(p^s)\cdots \varphi_{n-m+2}(p^s)\varphi_{n-m+1}(p^{s-j})p^{s\frac{m(m-1)}{2}-j(m-1)}p^{j-i}\\
		\times\binom{m-1+(j-i)-1}{j-i}_p
	\end{multline*}
	and for $j\geq m-1$ we get
	\begin{multline*}
		\tildeE_{i}(n\times m,p^s,p^j)=\varphi_n(p^s)p^{s(m-1)-i}\\
		\times \left(\E((n{-}1)\times (m{-}1),p^{s-1},p^{j-(m-1)})+\sum_{k=i}^{m-2}\tildeE((n-1)\times(m-1),p^s,p^j)\right)\\
		=\varphi_n(p^s)\cdots\varphi_{n-m+1}(p^{s-j})p^{s\frac{m(m-1)}{2}-j(m-1)-i}\binom{m-1+(j-i)-1}{j-i}_p,
	\end{multline*}
	again by using \eqref{eq:oldlemma}.
	
	Hence, no matter wish case the formulas for $\tildeE_i$ has the same form. Hence the formulas for $\tildeE_i$ are true for all $m$ and $j$.
	For $j<m$ we also get
	\begin{multline*}
		\E(n\times m,p^s,p^j)=\tildeE(n\times m,p^s,p^j)=\sum_{i=0}^{j}\tildeE_i(n\times m,p^s,p^j)\\
		=\varphi_{n}(p^s)\cdots \varphi_{n-m+2}(p^s)\varphi_{n-m+1}(p^{s-j})p^{s\frac{m(m-1)}{2}-j(m-1)}\\
		\times\left( \binom{m-2}{0}+p\binom{m-1}{1}+\ldots p^{j}\binom{m-1+j-1}{j}\right)\\
		=\varphi_{n}(p^s)\cdots \varphi_{n-m+2}(p^s)\varphi_{n-m+1}(p^{s-j})p^{s\frac{m(m-1)}{2}-j(m-1)}\binom{m+j-1}{j}_p.
	\end{multline*}
	So we have proved the formula for $\E$ for all $m$ and $j$ as long as  $j<m$.
	For $j\geq m$ we need to be more careful since
	\begin{equation}\label{eq:main:n1}
		\E(n\times m,p^s,p^j)=\E(n\times m,p^{s-1},p^{j-m})+\tildeE(n\times m,p^s,p^j).
	\end{equation}
	But, a second inductive reasoning solves the problem: 
	We know the formula for $\E$ is true for $0\leq j<m$. We can now use this result to get a well defined expression of
	$\E(n\times m,p^s,p^j)$  for $m\leq j< 2m$ by using \eqref{eq:main:n1} since $\tildeE$ is well defined. Then we repeat the process and get att well defined expression of $\E(n\times m,p^s,p^j)$ for $2m\leq j<3m$ and so on. 
	If we again use the properties of the $\varphi$-function (compare \eqref{eq:oldlemma}) we get
	\begin{multline*}
		\E(n\times m,p^{s-1},p^{j-m})=
		\varphi_{n}(p^{s-1})\cdots \varphi_{n-m+2}(p^{s-1})\varphi_{n-m+1}(p^{s-1-(j-m)})\\
		\times p^{(s-1)\frac{m(m-1)}{2}-(j-m)(m-1)}\binom{m+(j-m)-1}{j-m}\\
	=\varphi_{n}(p^s)\cdots\varphi_{n-m+2}(p^{s})\varphi_{n-m+1}(p^{s-j})p^{s\frac{m(m-1)}{2}-j(m-1)}\binom{j-1}{m-1}.
	\end{multline*}
	
	By using \eqref{eq:main:n1} we finally get
	
	\begin{multline*}
		\E(n\times m,p^s,p^j)\\
		=\varphi_n(p^s)\ldots \varphi_{n-m+2}(p^s)\varphi_{n-m+1}(p^{s-j})
		p^{s\frac{m(m-1)}{2}-j(m-1)}\binom{j-1}{m-1}_p\\
		+\sum_{i=0}^{j}\varphi_n(p^s)\cdots\varphi_{n-m+1}(p^{s-j})p^{s\frac{m(m-1)}{2}-j(m-1)-i}\binom{m-1+(j-i)-1}{j-i}_p\\
		=\varphi_{n}(p^s)\cdots \varphi_{n-m+2}(p^s)\varphi_{n-m+1}(p^{s-j})p^{s\frac{m(m-1)}{2}-j(m-1)}\binom{m+j-1}{j}_p.
	\end{multline*}	
\end{proof}
\section{Explicit formula for $j=s$}\label{sec5}
In this section we will prove an explicit formula for $\E(n\times m, p^s,p^j)$, when $j=s$. It is clear that $\E(n\times 1,p^s,p^s)=1$ since the only possible matrix is the zero matrix.
But, already for the $n\times 2$-case things start to be more complicated compared to the $j<s$ case. The formula can be stated in many different forms, but we will use the form that will be used in the general theorem later in this section.
\begin{proposition}\label{prop:nx2:sej}
	Consider $n\times 2$-matrices modulo $p^s$, where $n\geq 2$. The number of such matrices that have $p^s$ homogeneous solutions is
	\[
		\E(n\times 2,p^s,p^s)=\varphi_n(p^s)\varphi_{n-1}(1)p^{s-1}\binom{2}{1}_p
		+\varphi_{n}(p^{s-1})\varphi_{n-1}(p)p\binom{s-1}{1}_p.
	\]
\end{proposition}
\begin{proof}
	From Proposition \ref{prop:p_reduce} we get
	\[
		\E(n\times 2,p^s,p^s)=\E(n\times 2,p^{s-1},p^{s-2})+\tildeE(n\times 2,p^s,p^s).
	\]
	From Proposition \ref{prop:nx2:gen} we get
	\[
		\E(n\times 2,p^{s-1},p^{s-2})=\varphi_n(p^{s-1})\varphi_{n-1}(p)p\binom{s-1}{1}_p
	\]
	and finally by Proposition \ref{prop:nx2}
	\[
		\tildeE(n\times 2,p^s,p^s)=\varphi_n(p^s)(p^s+p^{s-1})=\varphi_n(p^s)p^{s-1}\binom{2}{1}_p.\qedhere
	\]
\end{proof}

We now consider the general $n\times m$ case where $n\geq m$.

\begin{theorem}\label{thm:main_slj}
Let $n\geq m$.
The number of $n\times m$-matrices that have $p^s$ homogeneous solutions modulo $p^s$ is
	\begin{equation}\label{eq1:thm:main_slj}
		\E(n\times m,p^s,p^s)=\sum_{k=0}^{m-1}\Phi(p^s,k)\binom{m}{k}_p\binom{s-1}{m-1-k}_p p^{a(m,k)s+b(m,k)},
	\end{equation}
	where
	\[
		\Phi(p^s,k)=\varphi_{n-(m-1)}(p^{m-1-k})\prod_{u=k}^{m-2}\varphi_{n-u}(p^{s-1})\prod_{v=0}^{k-1}\varphi_{n-v}(p^s)
	\]
	and
	\[
		a(m,k)=(\frac{m(m-1)}{2}-(m-1-k))\textrm{, }b(m,k)=\frac{(m-k)(m-3k-1)}{2}.
	\]
\end{theorem}

\begin{proof}
	From Proposition \ref{prop:p_reduce} we know
	\[
		\E(n\times m,p^s,p^s)=\tildeE(n\times m,p^s,p^s)+\E(n\times m, p^{s-1},p^{s-m})
	\]
	and from Theorem \ref{thm:main_sgj} we get
	\begin{multline*}
		\E(n\times m, p^{s-1},p^{s-m})\\
		=\varphi_n(p^{s-1})\cdots \varphi_{n-m+2}(p^{s-1})\varphi_{n-m+1}(p^{m-1})
		p^{s\frac{(m-1)(m-2)}{2}+\frac{m(m-1)}{2}}\binom{s-1}{m-1}_p\\
		=\Phi(p^s,0)p^{s a(m,0)+b(m,0)}\binom{m}{0}_p\binom{s-1}{m-1-0}.
	\end{multline*}
	Hence, $\E(n\times m, p^{s-1},p^{s-m})$ is the first term (for $k=0$) in the  sum of \eqref{eq1:thm:main_slj}. So we have left to prove that
	\begin{equation}\label{eq2:thm:main_slj}
		\tildeE(n\times m,p^s,p^s)=\sum_{k=1}^{m-1}\Phi(p^s,k)\binom{m}{k}_p\binom{s-1}{m-1-k}p^{a(m,k)s+b(m,k)}.
	\end{equation}
	

We will prove this by induction on $m$. In fact we will prove that
\begin{multline}\label{eq3:thm:main_slj}
	\tildeE_{m-k}(n\times m, p^s,p^s)\\
	=\sum_{i=1}^{\min(k,m-1)}\Phi_m(p^s,i) p^{s a(m,i)+b(m,i)+(k-i)}
	\binom{s-1}{m-(i+1)}_p\binom{k-1}{i-1}_p
\end{multline}
for $1\leq k\leq m$.

By the formula for $\tildeE_{m-k}$ in \eqref{eq3:thm:main_slj} we can find the desired formula for $\tildeE$ 
in \eqref{eq2:thm:main_slj}. 
In fact, 
\begin{multline*}
	\tildeE(n\times m,p^s,p^s)=\sum_{k=1}^{m}\tildeE_{m-k}(n\times m,p^s,p^s)\\
	=\sum_{k=1}^{m-1} \sum_{i=1}^k\Phi_m(p^s,i) p^{s a(m,i)+b(m,i)+(k-i)}
	\binom{s-1}{m-(i+1)}_p\binom{k-1}{i-1}_p\\
	+\sum_{i=1}^{m-1}\Phi_m(p^s,i) p^{s a(m,i)+b(m,i)+(m-i)}
	\binom{s-1}{m-(i+1)}_p\binom{m-1}{i-1}_p\\
	=\sum_{i=1}^{m-1}\sum_{k=i}^{m}\Phi_m(p^s,i) p^{s a(m,i)+b(m,i)}\binom{s-1}{m-(i+1)}_p p^{k-i}\binom{k-1}{i-1}_p\\
	=\sum_{i=1}^{m-1}\Phi_m(p^s,i) p^{s a(m,i)+b(m,i)}\binom{s-1}{m-(i+1)}_p
	\sum_{k=i}^{m}p^{k-i}\binom{k-1}{i-1}_p.
\end{multline*}
By the properties of the Gaussian coefficients, see \eqref{eq:q-bin_sum} in Section \ref{sec2}, we get
\begin{multline*}
	\sum_{k=i}^{m}p^{k-i}\binom{k-1}{i-1}_p=\sum_{d=0}^{m-i}p^d\binom{d+i-1}{i-1}_p
	\\
	=\sum_{d=0}^{m-i}p^d\binom{d+i-1}{d}_p=\binom{i-1+(m-i)+1}{m-i}_p=\binom{m}{i}_p.
\end{multline*}
Hence,
\[
\tildeE(n\times m,p^s,p^s)=\sum_{i=1}^{m-1}\Phi_m(p^s,i) p^{s a(m,i)+b(m,i)}\binom{s-1}{m-(i+1)}_p\binom{m}{i}_p.
\]
and the theorem will follow.

We now turn to the induction proof of \eqref{eq3:thm:main_slj}
For $m=2$ we have (see Proposition \ref{prop:nx2})
\begin{align*}
	\tildeE_{1}(n\times 2,p^s,p^s)=\varphi_n(p^s)p^{s-1}\\
	\tildeE_0(n\times 2,p^s,p^s)=\varphi_n(p^s)p^{s}
\end{align*}
and this corresponds to formulas in \eqref{eq3:thm:main_slj} for $m=2$, $k=1,2$.

Let our induction hypothesis be
\begin{multline*}
	\tildeE_{m-1-k}((n-1)\times (m-1),p^s,p^s)\\
	=\sum_{i=1}^k \Phi_{n-1,m-1}(p^s,i)p^{sa(m-1,i)+b(m-1,i)+k-i}
	\binom{s-1}{(m-1)-(i+1)}_p\binom{k-1}{i-1}_p
\end{multline*}
for $1\leq k\leq m-2$ and for $k=m-1$
\begin{multline*}
\tildeE_0((n-1)\times (m-1),p^s,p^s)\\
=\sum_{i=1}^{m-2} \Phi_{n-1,m-1}(p^s,i)p^{sa(m-1,i)+b(m-1,i)+m-1-i}
	\binom{s-1}{(m-1)-(i+1)}_p\binom{m-2}{i-1}_p.
\end{multline*}

From Proposition \ref{prop:recursion_i} and Theorem \ref{thm:main_sgj} we have
\begin{multline*}
	\tildeE_{m-1}(n\times m,p^s,p^s)=\varphi_n(p^s)(p^{s-1})^{m-1}\E((n-1)\times (m-1),p^{s-1},p^{s-(m-1)})\\
	=\varphi_n(p^s)(p^{s-1})^{m-1}\varphi_{n-1}(p^{s-1})\cdots \varphi_{n-m+2}(p^{s-1})\varphi_{n-m+1}(p^{m-2})\\
	\times p^{(s-1)\frac{(m-1)(m-2)}{2}-(s-(m-1))(m-2)}\binom{s-1}{s-(m-1)}_p\\
	=\Phi_{n,m}(p^s,1)p^{sa(m,1)+b(m,1)}\binom{s-1}{m-2}_p.
\end{multline*}
Again by Proposition \ref{prop:recursion_i} and the induction hypothesis we get for $2 \leq k\leq m-1$
\begin{multline*}
	\tildeE_{m-k}(n\times m,p^s,p^s)=\varphi_n(p^s) (p^{s-1})^{m-k}(p^s)^{k-1}\\
	\times \left(\E((n-1)\times (m-1),p^{s-1},p^{s-(m-1)})+\sum_{t=1}^{k-1}\tildeE_{m-1-t}((n-1)\times (m-1),p^s,p^s)\right)\\
	=\varphi_n(p^s) (p^{s-1})^{m-k}(p^s)^{k-1}\left(\Phi_{n-1,m-1}(p^s,0)p^{s a(m-1,0)+b(m-1,0)}\right.\\	
\left. +\sum_{t=1}^{k-1}\sum_{i=1}^t \Phi_{n-1,m-1}(p^s,i)p^{sa(m-1,i)+b(m-1,i)+(t-i)}
\binom{s-1}{m-1-(i+1)}_p \binom{t-1}{i-1}_p \right)\\
=\Phi_{n,m}(p^s,1)p^{s a(m,1)+b(m,1)+k-1}\binom{s-1}{m-2}_p\\
+\sum_{t=1}^{k-1}\sum_{i=1}^t 
\Phi_{n,m}(p^s,i+1)p^{s a(m,i+1)+b(m,i+1)+k-(i+1) +t-i}\binom{s-1}{m-1-(i+1)}_p \binom{t-1}{i-1}_p\\
=\Phi_{n,m}(p^s,1)p^{s a(m,1)+b(m,1)+k-1}\binom{s-1}{m-2}_p\\
+\sum_{i=1}^{k-1} \Phi_{n,m}(p^s,i+1)p^{s a(m,i+1)+b(m,i+1)+k-(i+1)}\binom{s-1}{m-1-(i+1)}_p
\sum_{t=i}^{k-1}p^{t-i}\binom{t-1}{i-1}_p.
\end{multline*}	
From \eqref{eq:q-bin_sum} in Section \ref{sec2} we now get
\[
	\sum_{t=i}^{k-1}p^{t-i}\binom{t-1}{i-1}_p=\binom{k-1}{i}_p
\]
and by letting $j=i+1$
\begin{multline*}
	\tildeE_{m-k}(n\times m,p^s,p^s)
	=\Phi_{n,m}(p^s,1)p^{s a(m,1)+b(m,1)+k-1}\binom{s-1}{m-2}_p+\\
	\sum_{j=2}^{k} \Phi_{n,m}(p^s,j)p^{s a(m,j)+b(m,j)+k-j}\binom{s-1}{m-1-j}_p\binom{s-1}{j-1}_p\\
	=\sum_{j=1}^{k} \Phi_{n,m}(p^s,j)p^{s a(m,j)+b(m,j)+k-j}\binom{s-1}{m-1-j}_p\binom{s-1}{j-1}_p.
\end{multline*}
A similar computation shows that
\[
	\tildeE_{0}(n\times m,p^s,p^s)=\sum_{j=1}^{m-1} \Phi_{n,m}(p^s,j)p^{s a(m,j)+b(m,j)+m-j}\binom{s-1}{m-1-j}_p\binom{s-1}{j-1}_p.
\]	
We have proved \eqref{eq2:thm:main_slj} by induction and the proof is complete.
\end{proof}
\section{Applications to counting the number of solutions}\label{sec6}
In this section we let $m=n$ and consider the quadratic system $A x\equiv 0\pmod{p^s}$.
Let $\eta(A,p^s)$ denote the number of solutions. In \cite{NilssonNyqvist2009} a formula for $\eta(A,p^s)$ is derived by using Gaussian elimination and reductions, in short the formula is
\[
	\eta(A,p^s)=\gcd(\det(A), p^{s+e}),
\]
for some non-negative integer $e$.

We can easily see that the formula $\eta(A,p^s)=\gcd(\det(A),p^s)$ is true if and only if $\eta(A,p^s)=p^j$ for some $j\leq s$. Hence, the probability that the number of homogeneous solutions to the random $n\times n$-matrix is given by 
$\gcd(\det(A),p^s)$ is
\[
	\P(\eta\leq p^s)=\frac{1}{p^{s n^2}}\sum_{j=0}^{s}\E(n\times n, p^s,p^j).
\]

We will prove in Theorem \ref{thm:main_sum} below that 
\begin{equation}\label{eq:main}
	\sum_{j=0}^{s}\E(n\times n, p^s,p^j)=p^{sn^2}(1-p^{-s-3}+\mathcal{O}(p^{-s-4})),
\end{equation}
where $\mathcal{O}$ is the usual big-Oh notation.
Hence, the probability that the number of solutions to $Ax\equiv 0\pmod{p^s}$ is less or equal to $p^s$ is
\[
	\P(\eta\leq p^s)=1-p^{-s-3}+\mathcal{O}(p^{-s-4}). 
\]

We will divide the proof of \eqref{eq:main} into two main parts. First we will estimate 
$\sum_{j=0}^{s-1}\E(n\times n, p^s,p^j)$ in Proposition \ref{prop:nxn_sgj} and then we estimate $\E(n\times n, p^s,p^s)$ in Proposition \ref{prop:nxn_slj}. We start with some notions that will be used throughout the section.
\begin{definition}
	We set for $1\leq k\leq n$, 
	\[
		S_{n,k}=\prod_{u=k}^n (p^u-1)
	\]
	and                                                                                            
\[
	A_{n,k}=n+(n+1)+\ldots + k=\frac{(n-k+1)(n+k)}{2}.
\]
We also set $A_n=A_{n,1}$ and $S_n=S_{n,1}$.
\end{definition}

We will need estimates of the products of the generalized Euler functions that appear in Theorem \ref{thm:main_sgj} and Theorem \ref{thm:main_slj}.
\begin{lemma}\label{lma:nxn_sgj}
		We have for $j< s$
		\[
			\varphi_1(p^{s-j})\prod_{u=2}^n \varphi_u(p^s)=p^{(s-1)A_n - j}S_{n,1}.
		\]
\end{lemma}
\begin{proof}
From the definition of $\varphi$ we get
	\begin{multline*}
		\varphi_1(p^{s-j})\prod_{u=2}^n \varphi_u(p^s)
			=p^{s-j-1}(p-1)\prod_{u=2}^n (p^{su}-p^{(s-1)u})\\
			=p^{s-j-1}\prod_{u=2}^n p^{(s-1)u}\prod_{u=1}^n (p^u-1)
			=p^{(s-1)A_{n}-j}S_{n,1}.
	\end{multline*}
\end{proof}

Recall that $\Phi(p^s,k)$ in Theorem \ref{thm:main_slj} is defined as
\[
	\Phi(p^s,k)=\varphi_{n-(m-1)}(p^{m-1-k})\prod_{u=k}^{m-2}\varphi_{n-u}(p^{s-1})\prod_{v=0}^{k-1}\varphi_{n-v}(p^s).
\]
For $n=m$ we get
\begin{equation}\label{eq:Phi_nlm}
	\Phi(p^s,k)=\varphi_{1}(p^{n-1-k})\prod_{u=k}^{n-2}\varphi_{n-u}(p^{s-1})\prod_{v=0}^{k-1}\varphi_{n-v}(p^s).
\end{equation}

\begin{lemma}\label{lma:nxn_slj}
	We have
	\[
		\Phi_{n}(p^s,k)=
		\begin{cases}
			p^{s A_{n,2}-2 A_{n,1}}S_{n,1}, \textrm{ if } k=0\\
			p^{s A_{n-2}-A_n - A_{n-k-1}}S_{n,1},\textrm{ if } 1\leq k\leq n-2\\
			p^{(s-1)A_{n,2}} S_{n,2},\textrm{ if } k=n-1.
		\end{cases}
	\]
\end{lemma}
\begin{proof}
For $k=0$ we get from \eqref{eq:Phi_nlm} that
	\begin{multline*}
		\Phi_{n}(p^s,0)=\varphi_1(p^{n-1})\prod_{a=2}^{n} \varphi_a(p^{s-1})\\
		=p^{n-2}(p-1)\prod_{a=2}^n (p^{(s-1)a}-p^{(s-2)a})
		=p^{n-2}\prod_{a=2}^n p^{(s-2)a} S_{n,1}\\
		=p^{s A_{n,2}-2 A_{n,1}}S_{n,1}.
	\end{multline*}
	For $1\leq k\leq n-2$,
	\begin{multline*}
		\Phi(p^s,k)=\varphi_1(p^{n-1-k})\prod_{a=2}^{n-k}\varphi_a(p^{s-1})\prod_{b=n-k+1}^n \varphi_b(p^s)\\
		=p^{n-k-2}(p-1)\prod_{a=2}^{n-k}(p^{(s-1)a}-p^{(s-2)a})\prod_{b=n-k+1}^n (p^{sb}-p^{(s-1)b})\\  
		=p^{n-k-2}p^{(s-2)A_{n-k,2}}p^{(s-1)A_{n,n-k+1}}S_{n,1}\\
		=p^{s A_{n,2}-2 A_{n-k,1}-A_{n,n-k}+n-k} S_{n,1}
		=p^{s A_{n-2}-A_n - A_{n-k-1}}S_{n,1}.
	\end{multline*}
	Finally, for $k=n-1$,
	\begin{multline*}
		\Phi_{n}(p^s,n-1)=\prod_{b=2}^n\varphi_b(p^s)=\prod_{b=2}^n (p^{sb}-p^{(s-1)b})\\
		=p^{(s-1)A_{n,2}}\prod_{b=2}^n(p^b-1)=p^{(s-1)A_{n,2}} S_{n,2}.
	\end{multline*}
\end{proof}
We will also need an expansion of $S_{n,k}$. 
\begin{lemma}\label{nxn:Snk}
	Let $n\geq k \geq 1$ be integers. We have
	\begin{equation}\label{eq:gen_bin}
		\prod_{u=k}^n(p^u-1)=\sum_{j=0}^{n-k+1}(-1)^j p^{A_{n-j,k}}\binom{n-k+1}{j}_p,
	\end{equation}
	where for $0\leq j\leq n-k$
	\[
		A_{n-j,k}=(n-j)+(n-j-1)+\ldots+ (k+1)+ k
	\]
	and for $j=n-k+1$ we let $A_{k-1,k}=0$.
\end{lemma}
\begin{proof}
	We prove this by downwards induction on $k$.
	For $k=n$ the right-hand side of \eqref{eq:gen_bin} is
	\[
		\sum_{j=0}^{1}(-1)^j p^{A_{n-j,n}}\binom{1}{j}_p=p^n-1.
	\]
	Hence true for $k=n$.
	Assume now that true for some specific $k$.
	\begin{multline}
		\prod_{u=k-1}^n(p^u-1)=(p^{k-1}-1)\prod_{u=k}^n (p^u-1)\\
		=(p^{k-1}-1)\sum_{j=0}^{n-k+1}(-1)^j p^{A_{n-j,k}}\binom{n-k+1}{j}_p\\
		\label{eq:sum_gen_bin_2}
		=\sum_{j=0}^{n-k+1}(-1)^j p^{A_{n-j,k}+(k-1)}\binom{n-k+1}{j}_p\\
		-\sum_{j=0}^{n-k+1}(-1)^j p^{A_{n-j,k}}\binom{n-k+1}{j}_p.
	\end{multline}
	If we note that
	\begin{multline*}
		A_{n-j,k}=(n-j)+(n-j-1)+\ldots + k\\
		=(n-(j+1))+ \ldots + k + (k-1) +[n-j-k+1]
	\end{multline*}
	we can rewrite the second sum as
	\begin{multline*}
		\sum_{j=0}^{n-k+1}(-1)^j p^{A_{n-j,k}}\binom{n-k+1}{j}_p\\
		=\sum_{j=0}^{n-k+1}(-1)^j p^{A_{n-(j+1),k-1}}p^{n-j-k+1}\binom{n-k+1}{j}_p\\
		=\sum_{i=1}^{n-k+2}(-1)^{i-1}p^{A_{n-i,k-1}}p^{n-j-k+2}\binom{n-k+1}{i-1}_p.
	\end{multline*}
	Going back to \eqref{eq:sum_gen_bin_2} we get
	\begin{multline*}
		\prod_{u=k-1}^n(p^u-1)
		=p^{A_{n,k-1}}+(-1)^{n-k+2}\\
		+\left(\sum_{j=1}^{n-k+1}(-1)^{j}p^{A_{n-j,k-1}}\left[\binom{n-k+1}{j}+p^{n-j-k+2}\binom{n-k+1}{j-1}\right]\right) 
		\\
		=\sum_{j=0}^{n-k+2} (-1)^j p^{A_{n-j,k-1}}\binom{n-k+2}{j}_p.
	\end{multline*}
	This the the right-hand side of \eqref{eq:gen_bin} when $k$ is replaced by $k-1$.
\end{proof}
We can now get a simplified expression for $\sum_{j=0}^{s-1}\E(n\times n,p^s,p^j)$.
\begin{proposition}\label{prop:nxn_sgj}
With the same notations as before we have
\[
	\sum_{j=0}^{s-1}\E(n\times n,p^s,p^j)=p^{s(n^2-n)-A_{n-1}}S_{n+s-1,s}\\
	=p^{sn^2}\sum_{j=0}^{n}p^{-js-A_{j-1}}\binom{n}{j}_{1/p}.
\]
\end{proposition}
\begin{proof}
First note that
\[
	\binom{n+j-1}{j}_p = p^{(n-1)j}\binom{n+j-1}{j}_{1/p}.
\]
From Theorem \ref{thm:main_sgj} and Lemma \ref{lma:nxn_sgj} we get
\begin{multline*}
	\sum_{j=0}^{s-1}\E(n\times n,p^s,p^j)=\sum_{j=0}^{s-1}p^{sn^2-A_n-j n}S_{n,1}\binom{n+j-1}{j}_p\\
	=p^{s n^2 -A_n} S_{n,1}\sum_{j=0}^{s-1}p^{-jn}\binom{n+j-1}{j}_p
	=p^{s n^2 -A_n} S_{n,1}\sum_{j=0}^{s-1}(1/p)^{j}\binom{n+j-1}{j}_{1/p}\\
	=p^{s n^2 -A_n} S_{n,1}\binom{n+s-1}{s-1}_{1/p}.
\end{multline*}
Here we used \eqref{eq:q-bin_sum} from Section \ref{sec2}  in the last step. Moreover,
\begin{multline*}
	p^{s n^2 -A_n} S_{n,1}\binom{n+s-1}{s-1}_{1/p}
	=p^{s n^2 -A_n} S_{n,1}p^{-(s-1)n}\binom{n+s-1}{n}_{p}\\
	=p^{s n^2-s n -A_{n-1}} S_{n+s-1,s}.
\end{multline*}
By Lemma \ref{nxn:Snk} we get
\[
	S_{n+s-1,s}=\sum_{j=0}^n (-1)^j p^{A_{n+s-1-j,2}}\binom{n}{j}_p.
\]
Observing that 
\[
	A_{n+s-1-j,s}=A_{n-(j+1)}+s(n-j)
\]
for $0\leq j\leq n-1$ and setting $A_{s-1,s}=0$.
we get
\begin{multline*}
	\sum_{j=0}^{s-1}\E(n\times n,p^s,p^j)=p^{sn^2}\sum_{j=0}^n(-1)^j p^{-sn-A_{n-1}}p^{A_{n-(j+1)}+s(n-j)}\binom{n}{j}_p\\
	=p^{sn^2}\sum_{j=0}^n(-1)^j p^{-A_{j-1}-s j}\binom{n}{j}_{1/p},
\end{multline*}
by using that $\binom{n}{j}_p=p^{j(n-j)}\binom{n}{j}_{1/p}$.
\end{proof}

We now try to find a simplified expression for $\E(n\times n, p^s, p^s)$.
\begin{lemma}\label{lma:nxn_Vandermonde}
	We have
	\[
		\sum_{k=\max(0,n-s)}^{n-1}p^{k(s-n+k)}\binom{n}{k}_p\binom{s-1}{n-1-k}_p=\binom{n+s-1}{s}_p.
	\]
\end{lemma}
\begin{proof}
	We re-write the left-hand side as
	\begin{multline*}
		\sum_{i=1}^s p^{(n-i)(s-i)}\binom{n}{n-i}_p\binom{s-1}{i-1}_p
		=\sum_{i=1}^s p^{(n-i)(s-i)}\binom{n}{i}_p\binom{s-1}{s-i}_p\\
		=\sum_{i=0}^s p^{(n-i)(s-i)}\binom{n}{i}_p\binom{s-1}{s-i}_p,
	\end{multline*}
	since $\binom{s-1}{s}_p=0$.
	The so called $q$-Vandermonde identity (also stated in \eqref{eq:q_bin_prod})
	\[
		\sum_{k=0}^h\binom{n}{k}_q\binom{m}{h-k}_q q^{(n-k)(h-k)}=\binom{m+n}{h}_q,
	\]
	can now be used. Let $q=p$, $m=s-1$ and $h=s$ in this formula then
	\[
	\sum_{i=0}^s p^{(n-i)(s-i)}\binom{n}{i}_p\binom{s-1}{s-i}_p=\binom{s-1 +n}{s}_p.
	\]
\end{proof}
\begin{proposition}\label{prop:nxn_slj}
	With the above notations
	\begin{multline*}
		 \frac{\E(n\times n,p^s,p^s)}{p^{sn^2}}=(p^n-1)\sum_{j=0}^{n-1}(-1)^j p^{-n-(j+1)s-A_j}\binom{n-1}{j}_{1/p}\\
		 + p^{-s}\binom{n}{1}_{1/p}\sum_{i=0}^{n-1}(-1)^i p^{-A_{i+1}}\binom{n-1}{i}_{1/p}.
	\end{multline*}
\end{proposition}
\begin{proof}
	For $k=0$ we get the term
	\begin{multline*}
		\Phi(p^s,0)\binom{n}{0}_p\binom{s-1}{n-1}_p p^{s A_{n-2}+A_{n-1}}=
		p^{s A_{n,2}-2 A_n +n}p^{s A_{n-2}+A_{n-1}}S_{n,1}\binom{s-1}{n-1}_p\\
		=p^{s(A_{n,2}+A_{n-2})- A_n}S_{n,1}\binom{s-1}{n-1}_p\\
		=p^{s(n^2-n)-A_n}S_{n,1}\binom{s-1}{n-1}_p
	\end{multline*}
	For $1\leq k \leq n-2$ we get from Lemma \ref{lma:nxn_slj} that
	\begin{multline*}
		\Phi(p^s,k)\binom{n}{k}_p\binom{s-1}{n-1-k}_p p^{s a(n,k)+ b(n,k)}\\
		=p^{s A_{n,2}-A_n-A_{n-k-1}}S_{n,1}\binom{n}{k}_p\binom{s-1}{n-1-k}_p 
		p^{s (A_{n-2}+k) + A_{n-k-1}-k(n-k)}\\
		=p^{s(n^2-n+k)-A_n-k(n-k)}S_{n,1}\binom{n}{k}_p\binom{s-1}{n-1-k}_p\\
		=p^{s(n^2-n)-A_n+k(s-n+k)}S_{n,1}\binom{n}{k}_p\binom{s-1}{n-1-k}_p.
	\end{multline*}
	
	Finally for $k=n-1$ we get
	\begin{multline*}
		\Phi(p^s,n-1)\binom{n}{n-1}_p\binom{s-1}{0}_p p^{sa(n,n-1)+b(n,n-1))}\\
		=p^{s A_{n,2} - A_{n,2}}S_{n,2}\binom{n}{1}_p p^{sA_{n-1}-(n-1)}
		=p^{s(n^2-n +(n-1))-A_{n}-(n-1)}p S_{n,2}\binom{n}{1}_p.
	\end{multline*}
	We get
	\begin{multline*}
		\E(n\times n,p^s,p^s)=p^{s(n^2-n)-A_n}S_{n,2}\left( (p-1)\sum_{k=0}^{n-2}p^{k(s-n+k)}\binom{n}{k}_p\binom{s-1}{n-1-k}_p\right.\\
		\left.+(p-1) \cdot p^{s(n-1)-(n-1)}\binom{n}{1}_p\binom{s-1}{0}_p-p^{s(n-1)-(n-1)}\binom{n}{1}_p\binom{s-1}{0}_p\right)\\
		=p^{s(n^2-n)-A_n}\left(S_{n,1}\sum_{k=0}^{n-1}p^{k(s-n+k)}\binom{n}{k}_p\binom{s-1}{n-1-k}_p
		+ S_{n,2}p^{(s-1)(n-1)}\binom{n}{1}_p\right)\\
		=p^{s n^2-A_n}\left(S_{n,1}\binom{n+s-1}{s}_p+S_{n,2}p^{(s-1)(n-1)}\binom{n}{1}_p\right),
	\end{multline*}
	by Lemma \ref{lma:nxn_Vandermonde}. Moreover,
	\[
		S_{n,1}\binom{n+s-1}{s}_p=S_{n,1}\binom{n+s-1}{n-1}=(p^n-1)S_{n+s-1,s+1}
	\]
	and by Lemma \ref{nxn:Snk} we get
	\begin{multline}
		\E(n\times n,p^s,p^s)=p^{s(n^2-n)-A_n}\left((p^n-1)\sum_{j=0}^{n-1}(-1)^j p^{A_{n+s-1-j,s+1}}\binom{n-1}{j}_p\right.\\
		+\left.p^{(s-1)(n-1)}\binom{n}{1}_p\sum_{i=0}^{n-1}(-1)^i p^{A_{n-i,2}}\binom{n-1}{i}_p \right)\\
		=p^{s(n^2-n)-A_n}\left((p^n-1)\sum_{j=0}^{n-1}(-1)^j p^{A_{n+s-1-j,s+1}}p^{j(n-1-j)}\binom{n-1}{j}_{1/p} \right.\\
		\left.+ p^{(s-1)(n-1)}p^{n-1}\binom{n}{1}_{1/p}\sum_{i=0}^{n-1}(-1)^i p^{A_{n-i,2}}p^{i(n-1-i)}\binom{n-1}{i}_{1/p} \right).\label{eq:nn1}
	\end{multline}
	Noting that
	\begin{multline*}
		-sn-A_n+A_{n+s-1-j,s+1}-j(n-1-j)\\=-sn-A_n+A_{n-j-1}+ns-(j+1)s -j(n-1-j)
		=-n-(j+1)s-A_j
	\end{multline*}
	and
	\begin{multline*}
	 -sn -A_n + (s-1)(n-1) +(n-1)+A_{n-i,2} +i(n-1-i)\\=-s-n+1 +(n-1)-ni + A_{i-1}-1 +i n-2A_i
	=-s-A_{i+1}
	\end{multline*}
	we can simplify \eqref{eq:nn1} to
	\begin{multline}
		\E(n\times n, p^s,p^s)=p^{sn^2} \left((p^n-1)\sum_{j=0}^{n-1}(-1)^j p^{-n-(j+1)s-A_j}\binom{n-1}{j}_{1/p}\right.\\
		\left. + p^{-s}\binom{n}{1}_{1/p}\sum_{i=0}^{n-1}(-1)^i p^{-A_{i+1}}\binom{n-1}{i}_{1/p}\right).
	\end{multline}
	\end{proof}

We are now ready to prove the main result of this section.
\begin{theorem}\label{thm:main_sum}
The number of $n\times n$-matrices $A$ over $\Z_{p^s}$ such that $Ax=0$ has at most $p^s$ solutions is 
\[
	\sum_{j=0}^{s}\E(n\times n,p^s,p^j)=p^{sn^2}(1- p^{-s-3}+\mathcal{O}(p^{-s-4})).
\]
If the elements in the matrix are chosen uniformly the probability that the number of solutions $\eta$ is at most $p^s$ is
\[
	\P(\eta\leq p^s)=1- p^{-s-3}+\mathcal{O}(p^{-s-4}).
\] 
\end{theorem}
\begin{proof}
	In the estimations below we have used $s\geq 1$ and $n\geq 2$.
	From Proposition \ref{prop:nxn_sgj} we have
	\begin{equation}
		\label{eq:final1}
		\frac{1}{p^{sn^2}}\sum_{j=0}^{s}\E(n\times n,p^s,p^j)=1-p^{-s}\binom{n}{1}_{1/p}+p^{-2s-1}\binom{n}{2}_{1/p}+\mathcal{O}(p^{-s-5}).
	\end{equation}
	From Proposition \ref{prop:nxn_slj} we get 
	\begin{multline}\label{eq:final2}
		\frac{\E(n\times n, p^s,p^s)}{p^{sn^2}}=(p^n-1)\sum_{j=0}^{n-1}(-1)^j p^{-n-(j+1)s-A_j}\binom{n-1}{j}_{1/p}\\
		 + p^{-s}\binom{n}{1}_{1/p}\sum_{i=0}^{n-1}(-1)^i p^{-A_{i+1}}\binom{n-1}{i}_{1/p}\\
		=\sum_{j=0}^{n-1}(-1)^jp^{-(j+1)s-A_j}\binom{n-1}{1}_{1/p}-p^n\sum_{j=0}^{n-1}(-1)^jp^{-(j+1)s-A_j}\binom{n-1}{1}_{1/p}\\
		+p^{-s}(1+p^{-1}+\ldots +p^{-n+1})\left(p^{-1}-p^{-3}\binom{n-1}{1}_{1/p}\right) +\mathcal{O}(p^{-s-4})\\
		=p^{-s}-p^{-2s-1}\binom{n-1}{1}_{1/p} -p^{-n-s}+p^{-s-1}(1+p^{-1}+\ldots +p^{-n+1})\\
		-p^{-s-3}(1+p^{-1}+\ldots +p^{-n+1})(1+p^{-1}+\ldots + p^{-n+2}) +\mathcal{O}(p^{-s-4})\\
		=p^{-s}\binom{n}{1}_{1/p}-p^{-2s-1}\binom{n-1}{1}_{1/p}-p^{-s-3}+\mathcal{O}(p^{s-4}).
	\end{multline}
	
	By adding \eqref{eq:final1} and \eqref{eq:final2} we finally get
	\begin{multline*}
	 1+p^{-2s-1}\left(\binom{n}{2}_{1/p}-\binom{n-1}{1}_{1/p}\right)-p^{-s-3}+\mathcal{O}(p^{-s-4})\\
	=1-p^{-s-3}+\mathcal{O}(p^{-s-4})
	\end{multline*}
	since
	\[
		\binom{n}{2}_{1/p}-\binom{n-1}{1}_{1/p}=1+p^{-1}+2 p^{-2}+\ldots -1-p^{-1}-p^{-2}-\ldots=p^{-2}+\mathcal{O}(p^{-3})
	\]
	and $2s+3\geq s+4$ for $s\geq 1$.
\end{proof}
\begin{corollary}
	The probability that the formula $\gcd(\det(A),p^s)$ gives the number of solutions to $Ax\equiv 0\pmod{p^s}$ is
	$1- p^{-s-3}+\mathcal{O}(p^{-s-4}).$
\end{corollary}
This follows directly from the theorem and the discussion in the beginning of this section.

\section{Discussion}

In this paper we have considered system of equations $Ax=0$ over $\Z_{p^s}$, where $A$ is an $n\times m$-matrix. Especially we have been interested in finding the number of matrices $\E(n\times m,p^s,p^j)$ such that the system have $p^j$ solutions for $0\leq j\leq s$.
We have found different recursive relations for $\E$ as well as explicit formulas for the case $n\geq m$. 
In fact, the recursive relations still hold in the $n< m$ case, but the author has not yet found nice explicit formulas in the general case. However, one could easily prove that
\[
	\E(1\times m, p^s,p^{s(m-1)+r})=\varphi(p^{s-r})p^{(s-r-1)(m-1)}\binom{m}{1}_p
\]
for $0\leq r\leq s-1$.

The author has nor been able to find simple explicit formulas for the $j>s$ case. 
Again the recursive formulas in Section \ref{sec3} can be used to compute $\E(n\times m,p^s,p^j)$ for explicit values of the parameters. But, general expressions seem to be complicated. Note also that for the $n<m$ case, the number of solutions is always 
greater than $p^s$, so $j>s$ is the only interesting case. 
In this paper we only considered the system of equations modulo a prime power. But by using the Chinese Remainder Theorem we can easily get results for a general modulus $N$. Let $N=\prod p_i^{s_i}$, where $p_i$ different primes. Then
\[
	\E(n\times m,\prod_i p_i^{s_i},\prod_i p_i^{j_i})=\prod_i \E(n\times m,p_i^{s_i},p_i^{j_i}).
\]
Note that the number of possible solutions to $Ax\equiv 0\pmod{N}$ is $\prod_i p^{j_i}$, where
$0\leq j_i\leq s_i m$.

Regarding the results from Section \ref{sec6}. Also here it is possible to somewhat generalize the results to a general modulus $N$. Note that
\[
\gcd(\det(A),\prod_i p_i^{s_i})=\prod_i \gcd(\det(A),p^{s_i})
\]
and that the number of solutions 
\[
	\eta(A,\prod_i p_i^{s_i})=\prod_i\eta(A,p_i^{s_i})=\prod_i\gcd(\det(A),p^{s_i+e_i}),
\]
for some non-negative integers $e_i$.
The formula $\gcd(\det(A),N)$ gives the number of solutions if and only if all the reduced systems modulo $p_i^{s_i}$ have at most $p_i^{s_i}$ solutions. Heuristically the probability that the formula is correct is something like
$1-\sum_i p_i^{-s_i-3}$ but the true estimate depends however much on the number of prime factors and the size of them.

\end{document}